\documentclass[12pt]{amsart}
\usepackage{amsmath, amssymb, amsthm, latexsym}
\usepackage{tikz-cd}
\usepackage{amssymb,amscd,amsmath}
\usepackage{latexsym}
\usepackage[center]{caption}
\usepackage{tikz}
\usepackage{tikz}
\usepackage{mathrsfs}
\newcounter{braid}
\newcounter{strands}

\DeclareMathAlphabet{\bsf}{OT1}{cmss}{bx}{n}

\pgfkeyssetvalue{/tikz/braid height}{1cm}
\pgfkeyssetvalue{/tikz/braid width}{1cm}
\pgfkeyssetvalue{/tikz/braid start}{(0,0)}
\pgfkeyssetvalue{/tikz/braid colour}{black}
\pgfkeys{/tikz/strands/.code={\setcounter{strands}{#1}}}

\makeatletter
\def\cross{%
  \@ifnextchar^{\message{Got sup}\cross@sup}{\cross@sub}}

\def\cross@sup^#1_#2{\render@cross{#2}{#1}}

\def\cross@sub_#1{\@ifnextchar^{\cross@@sub{#1}}{\render@cross{#1}{1}}}

\def\cross@@sub#1^#2{\render@cross{#1}{#2}}

\def\render@cross#1#2{
  \def\strand{#1}
  \def\crossing{#2}
  \pgfmathsetmacro{\cross@y}{-\value{braid}*\braid@h}
  \pgfmathtruncatemacro{\nextstrand}{#1+1}
  \foreach \thread in {1,...,\value{strands}}
  {
    \pgfmathsetmacro{\strand@x}{\thread * \braid@w}
    \ifnum\thread=\strand
    \pgfmathsetmacro{\over@x}{\strand * \braid@w + .5*(1 - \crossing) * \braid@w}
    \pgfmathsetmacro{\under@x}{\strand * \braid@w + .5*(1 + \crossing) * \braid@w}
    \draw[braid] \pgfkeysvalueof{/tikz/braid start} +(\under@x pt,\cross@y pt) to[out=-90,in=90] +(\over@x pt,\cross@y pt -\braid@h);
    \draw[braid] \pgfkeysvalueof{/tikz/braid start} +(\over@x pt,\cross@y pt) to[out=-90,in=90] +(\under@x pt,\cross@y pt -\braid@h);
    \else
    \ifnum\thread=\nextstrand
    \else
     \draw[braid] \pgfkeysvalueof{/tikz/braid start} ++(\strand@x pt,\cross@y pt) -- ++(0,-\braid@h);
    \fi
   \fi
  }
  \stepcounter{braid}
}

\tikzset{braid/.style={double=\pgfkeysvalueof{/tikz/braid colour},double distance=1pt,line width=2pt,white}}

\newcommand{\braid}[2][]{%
  \begingroup
  \pgfkeys{/tikz/strands=2}
  \tikzset{#1}
  \pgfkeysgetvalue{/tikz/braid width}{\braid@w}
  \pgfkeysgetvalue{/tikz/braid height}{\braid@h}
  \setcounter{braid}{0}
  \let\sigma=\cross
  #2
  \endgroup
}
\makeatother

\input xypic
\newtheorem{theorem}{Theorem}
\newtheorem{proposition}[theorem]{Proposition}

\newtheorem{lemma}[theorem]{Lemma}

\newtheorem{corollary}[theorem]{Corollary}
\newtheorem{definition}[theorem]{Definition}


\def\Z{\mathbb{Z}}

\def\md{\mathcal{D}}

\def\qed{\hfill$\square$\medskip}

\def\Zpk{\mathbb{Z}/p^{k}}
\def\Zpk1{\mathbb{Z}/p^{k-1}}

\newcommand{\rref}[1]{(\ref{#1})}

\newcommand{\beg}[2]{\begin{equation}\label{#1}#2\end{equation}}

\def\sl2{\widetilde{SL_{2}(\Z)}}

\def\md
\def\rank{\operatorname{rank}}

\title[]{Some remarks on Mackey functors}
\author{Sophie Kriz}

\begin{document}
\maketitle

\begin{abstract}

The purpose of this paper is mainly to record how certain homotopy-theoretical constructions on ordinary
$G$-equivariant cohomology spectra $HM$ for a Mackey functor $M$, in particular products and duality,
can be described on chain level. We will also discuss certain facts about modules over the constant Green functor 
$\underline{\Z}$.

\end{abstract}

\vspace*{10mm}
\tableofcontents

\vspace{5mm}

\section{Introduction}

\vspace{5mm}

Ordinary $RO(G)$-graded equivariant homology and cohomology for a compact Lie group was
defined by Lewis, May, and McClure in \cite{LewisMayMcClureOrdinary}.
This theory is thus represented by a $G$-equivariant spectrum in the sense of \cite{LewisMaySteinbergerEquivariant}.
For $G$ finite, however, equivariant (co)homology is also deeply connected with Mackey functors \cite{DressMackeyFunctors, GreenleesMackeyFunctors} and can be
characterized entirely on chain level.
As a beginning of the story, ordinary equivariant (co)homology is a special case of Bredon cohomology \cite{Bredon}, but more needs to be said to capture
the whole Mackey functor structure.
The purpose of this paper is to examine this story in more detail and show how certain constructions on
$G$-equivariant spectra can be done on chain level.
While much of the material presented here is ``standard" (with the exception, perhaps, of the last section),
some of it may not be easy to find in the literature.

Given a $G$-CW-complex $X$, we have the equivariant coefficient system-valued cellular chain complex $C_G(X)$ (see Section
\ref{MackeyFunctorSection}). Equivariant cohomology
(resp. homology) of $X$ with respect to a coefficient system (resp. a co-coefficient system) $M$ is the cohomology (resp. homology) of
\beg{EquivariantCoHomologyDefn}{Hom_{\mathscr{O}_G} (C_G(X), M) \text{ (resp. } C_G (X) \otimes_{\mathscr{O}_G} M).}
Of course, both constructions apply when $M$ is a Mackey functor, which is our main case of interest.
(The more general case corresponds to $\Z$-graded Bredon equivariant cohomology and homology theory \cite{Bredon}.)

One point is to clarify how this structure behaves under products.
From a spectral point of view, constructions \rref{EquivariantCoHomologyDefn} correspond to
$$F(X_+, HM), \; \; X_+ \wedge HM,$$
so the discussion of products reduces to multiplicative properties of equivariant Eilenberg-Mac Lane spectra \cite{LewisMayMcClureOrdinary}.

On the chain level, there is a tensor product $\boxtimes$ of coefficient systems such that
$$C_G (X\times Y) \cong C_G (X) \boxtimes C_G (Y).$$
We discuss this in Section \ref{CoefficientSystemLevel}. On the level of spectra, equivariant Eilenberg-Mac Lane spectra with coefficients
in a Mackey functor \cite{LewisMayMcClureOrdinary} are (rigid) module spectra
over the $E_{\infty}$ ring spectrum $H\mathscr{A}$ where $\mathscr{A}$ is the universal (Burnside ring) Green functor.
Thus, for Mackey functors $M, N$ we have the smash product
\beg{SmashProductFromMackeyFunctors}{HM \wedge_{H\mathscr{A}} MN.}
This is reflected in Mackey functors by the box-product $\Box$ (see \cite{DressMackeyFunctors, LiBoxProduct} and Section \ref{MackeyFunctorSection} below), and in fact,
\rref{SmashProductFromMackeyFunctors} is equivalent to the total left derived functor of $H(M\Box N)$.

We shall also discuss duality.
Starting from $G$-CW-complexes, the coefficient system-valued chain complex $C_G(X)$ is not sufficient, 
and we need to discuss what is the appropriate Mackey functor-valued chain complex $C_M (X)$. 
The right construction, which also carries $\boxtimes$ to $\Box$,
turns out to be the left Kan extension from coefficient systems to Mackey functors, which we discuss in Section \ref{MackeyFunctorSection}.

Duality is discussed in Section \ref{DualitySection}. Mackey functors form a closed symmetric monoidal category
with an internal Hom functor $Hom_M$. For a $G$-CW-complex $X$,
$Hom_M (C_M X,\mathscr{A})$ is the right Kan extension from co-coefficient systems to Mackey functors applied to
$$C_G^* X = Hom (C_G X, \Z).$$

We conclude with Section \ref{AppendixMackeyChains}, where we discuss the role of modules over the constant Green functor $\underline{\Z}$.
An example is
the Mackey functor
\beg{IntroductionFixedPointsMackeyFunctor}{C_{\underline{\Z}} X: G/ H \mapsto (CX)^H}
for a $G$-CW complex $X$.
Other aspects of equivariant homology with constant coefficients were discussed in a previous paper \cite{GeomFixedPoints}.
The facts presented here are, however, different.
One reason for the significance of Section \ref{AppendixMackeyChains} in the present story
is that it clarifies why we cannot just work with the fixed points of the ordinary chain complex as a chain-level model
of general equivariant homology with Mackey functor coefficients.
However, Mackey $\underline{\Z}$-modules are of independent interest and in some sense, can be
considered as an alternative type of representation theory.

In fact, \rref{IntroductionFixedPointsMackeyFunctor} turns out to correspond to the spectrum $X_+ \wedge H\underline{\Z}$, i.e. the universal $\underline{\Z}$-module valued
Mackey chain complex on $X$, where $\underline{\Z}$ denotes the constant Green functor.
While the $H$-fixed points of a $\Z[G]$-module, for a subgroup $H\subseteq G$, always give a Mackey $\underline{\Z}$-module,
the converse is not always true (we give a characterization of Mackey $\underline{\Z}$-modules).
We discuss, however, a certain sense in which the converse is true on the derived level.
Finally, we briefly discuss cofixed points which give, in some sense, an equivalent theory
for $G$ cyclic, but not in general.
We give an example.

\vspace{5mm}

\section{Coefficient Systems and Products}\label{CoefficientSystemLevel}

\vspace{5mm}

For a CW-complex $X$, we denote by $C_*(X)$ the cellular chain complex of $X$. We denote by $\mathscr{O}_G$ the orbit category of $G$.
A $G$-coefficient system is a functor $\mathscr{O}_G^{Op} \rightarrow Ab $ (the category of abelian groups). Likewise, a $G$-co-coefficient
system is a functor $\mathscr{O}_G \rightarrow Ab$. 
The category of finite $G$-sets and $G$-maps will be denoted by
``f.$G$-Sets." For a $G$-CW complex $X$, let
$$C_G (X) (G/H) := C_G (X^H)$$
denote the cellular coefficient-system-valued chain complex of $X$.
On the other hand, let
$$C^*_G(X)(G/H) : = Hom_{Ab}(C_G(X)(G/H), \Z)$$
denote the dual of $C_G(X)$. Since $Hom_{Ab}$ is contravariant in the first variable,
$C_G^*(X)$ is a co-coefficient system.

A {\em Mackey functor} is a pair consisting of a coefficient system and a co-coefficient system which agree on objects.
There is a compatibility condition which will be discussed in Section \ref{MackeyFunctorSection}.

In this section, we will define a ``tensor product" $A \boxtimes B$ of coefficient systems
$$A, B:\mathscr{O}_G^{Op}\rightarrow Ab$$
First, for any coefficient system $A$, we can define an extension
$$A: \text{f.}G \text{-Sets} \rightarrow Ab$$
by
$$A(\coprod G/H_i) = \oplus A (G/H_i)$$

We have the Cartesian product functor
\beg{CartesianFiniteGSets}{\times :\text{f.} G\text{-Sets}^{Op}\times \text{f.} G\text{-Sets}^{Op} \rightarrow \text{f.}G \text{-Sets}^{Op}.}
Recall that for a finite $G$-set $S$, the representable functor by $S$ is a coefficient system
$$F_S : G/H \mapsto \Z Map_G (G/H, S)$$
which is the value at $\Z$ of the left adjoint to evaluation at $S$ (the value of the left adjoint on any abelian group $A$ is obtained by tensoring with $A$, 
i.e.
$$A \mapsto A\otimes F_S).$$

To see this adjunction, which is a variant of the Yoneda lemma,
first we have 
$$A\otimes F_S (U) = A\{ \text{f.}G\text{-Sets} (U , S)\}$$
for a finite $G$-set $U$.
Then, for an additive functor
$$\Phi: \text{f.} G\text{-Sets}^{Op} \rightarrow Ab,$$
it is enough to show that a homomorphism on abelian groups
\beg{AssumingFunctor}{h:A \rightarrow \Phi(S)}
uniquely determines a natural transformation
\beg{DeterminedFunctor}{H:A\otimes F_S \rightarrow \Phi.}
We have
$$H(S)(a\otimes (Id:S\rightarrow S)) = h(a)$$
for $a\in A$.
Now for any $G$-set $U$ we must define
$$H(U): A\otimes F_S(U) = A \otimes \{U\rightarrow S\} \rightarrow \Phi (U)$$
by putting, for an $a\in A$, and a $f: U\rightarrow S$,
$$H(U) (a\otimes f) = (\Phi(f)) \circ H(S) (a\otimes Id_S) = \Phi(f)(h(a)).$$

Now, given coefficient systems $A, B : \text{f.} G\text{-Sets}^{Op}\rightarrow Ab$, first define
$$A\underline{\boxtimes} B : \text{f.} G\text{-Sets}^{Op}\times \text{f.} G\text{-Sets}^{Op} \rightarrow Ab$$
by $A\underline{\boxtimes} B (S, T) := A(S) \otimes B(T)$.
We define $A\boxtimes B : \mathscr{O}_G^{Op} \rightarrow Ab$ as the left Kan extension of $A\underline{\otimes} B$
by the Cartesian product \rref{CartesianFiniteGSets}.

\begin{lemma}\label{CoefficientSystemBoxtimesConnection}
\begin{enumerate}
\item\label{BoxtimesConnectionFinite}
For finite $G$-sets $S, T$ we have
$$F_S \boxtimes F_T = F_{S\times T}.$$

\item\label{BoxtimesConnectionCWComplexes}
For $G$-CW complexes $X, Y$, we have
$$C_G(X) \boxtimes C_G (Y) = C_G (X\times Y).$$
\end{enumerate}

\end{lemma}

\begin{proof}
To prove \rref{BoxtimesConnectionFinite}, the left Kan extension along \rref{CartesianFiniteGSets}
$$Funct(\text{f.} G\text{-Sets}^{Op} \times \text{f.} G\text{-Sets}^{Op}, Ab)\rightarrow Funct(\text{f.} G\text{-Sets}^{Op}, Ab)$$
is defined to be the left adjoint of $Funct(\times, Ab)$.
Then we shall study the composition of functors
\beg{LemmaCompositionABFUNCTS}{Ab \rightarrow Funct(\text{f.} G\text{-Sets}^{Op} \times \text{f.} G\text{-Sets}^{Op}, Ab)\rightarrow Funct(\text{f.} G\text{-Sets}^{Op}, Ab),}
where the first functor is the left adjoint to evaluation at some fixed $(S, T)$.
The composition \rref{LemmaCompositionABFUNCTS} is left adjoint to the functor that evaluates a functor in $Funct(\text{f.} G\text{-Sets}^{Op}, Ab)$ to $ S\times T$.
Thus it sends $\Z$ to $F_{S\times T}$. On the other hand, the first functor \rref{LemmaCompositionABFUNCTS} sends $\Z$ to $F_S \underline{\boxtimes} F_T$.

As above, this means that a homomorphism
$$h:A \rightarrow \Phi(S, T)$$
uniquely determines a natural transformation
$$H:A\otimes (F_S \underline{\boxtimes} F_T) \rightarrow \Phi.$$
By definition, we have
$$F_S \underline{\boxtimes} F_T (U, V) = \Z \{ U\rightarrow S\} \otimes \Z \{ V\rightarrow T\} = \Z \{ (U\rightarrow S, V\rightarrow T)\}.$$
Again, we have
$$H(S,T) (a\otimes (Id:S\rightarrow S, Id:T\rightarrow T)) = h(a)$$
for $a\in A$.
Now, for any finite $G$-set $U$, we must define
$$H(U,V) : A \otimes (F_S \underline{\boxtimes} F_T (U,V)) \rightarrow \Phi (U,V)$$
by putting, for $a\in A$, $f_1: U\rightarrow S$, $f_2: V\rightarrow T$,
$$H(U,V) (a\otimes (f_1 \otimes f_2))= (\Phi (f_1, f_2)) \circ H(S, T) (a \otimes (Id_S\otimes Id_T))=$$
$$= (\Phi (f_1, f_2))(h(a)).$$
This concludes the proof of \rref{BoxtimesConnectionFinite}.

Finally, \rref{BoxtimesConnectionCWComplexes} is an immediate consequence, since
$$C_G(X)_n = F_{I_n}$$
where $I_n$ is the set of $n$-cells of $X$.

\end{proof}

\vspace{5mm}

\section{Mackey Functors and Products}\label{MackeyFunctorSection}

\vspace{5mm}

Ordinary equivariant $RO(G)$-graded homology and cohomology, however, work on the level
of Mackey functors, not coefficient systems.
This becomes important when one studies duality.
Mackey functors have their own product, different from the tensor product $\boxtimes$ of coefficient systems.
It is called the {\em box product}, and we denote it by $\Box$ (see \cite{DressMackeyFunctors, LiBoxProduct}).
The purpose of this section is to relate the results of the last section to statements about $\Box$.
We begin with a more detailed treatment of Mackey functors.

\vspace{3mm}

Recall the Burnside category $\mathscr{B}_G$. The objects of $\mathscr{B}_G$ are finite $G$-sets. 
Morphisms in $\mathscr{B}_G$ from $T$ to $U$
are given by the group completion of the set of isomorphism classes (in the $S$-coordinate) of diagrams
of finite $G$-sets
$$
\diagram
 & S\drto\dlto & \\
T & & U\\
\enddiagram
$$
with respect to the operation of coproducts (again in the $S$-coordinate).
The composition of two diagrams is defined in the obvious way using pullbacks (for details, see \cite{DressMackeyFunctors}).

Then a Mackey functor is defined to be an additive functor
$$\mathscr{B}_G \rightarrow Ab.$$
The category of Mackey functors over $G$ is denoted by $Mackey_G$.
(Note that by definition, $\mathscr{B}_G$ is self-dual.)
There is an operation for Mackey functors similar to the operation for coefficient systems $\boxtimes$ defined in the previous section.

This ``tensor product" of Mackey functors is denoted by $\Box$ (c.f. \cite{DressMackeyFunctors, LiBoxProduct}). 
We will review this construction.
First, as above in the previous section, given Mackey functors $M,N: \mathscr{B}_G\rightarrow Ab$,
define
\beg{UnderlinedBox}{M\underline{\Box} N : \mathscr{B}_G\times \mathscr{B}_G\rightarrow Ab}
by $M\underline{\Box} N (S, T) := M(S) \otimes N(T)$.
On the other hand, we also have a ``Cartesian product"
\beg{CartesianProductBurnside}{\times:\mathscr{B}_G \times \mathscr{B}_G \rightarrow \mathscr{B}_G.}
(Note that \rref{UnderlinedBox} and \rref{CartesianProductBurnside} are biadditive functors.
Also note that to define \rref{CartesianProductBurnside}, one must consider the group completion.)
Define $M\Box N$ to be the left Kan extension of $M\underline{\Box} N$ via \rref{CartesianProductBurnside}.

\vspace{5mm}

\noindent {\bf Comment:}

Commutative monoids $\mathscr{G}$ with respect to the box product 
$$\mathscr{G}\Box \mathscr{G} \rightarrow \mathscr{G}$$
are called Green functors.
A module $\mathscr{M}$ over a Green functor $\mathscr{G}$ is defined by
$$\mathscr{M} \Box \mathscr{G} \rightarrow \mathscr{M}$$
with the usual axioms.

Of course, we also must consider the unit of the box product. This is the universal Green functor $\mathscr{A}$, i.e. the Burnside ring functor. 
By the above method, it is the left Kan
extension of the functor
$$*\rightarrow Ab$$
$$* \mapsto \Z$$
via the inclusion
$$* \rightarrow \mathscr{B}_G$$
$$* \mapsto G/G.$$
Then $\mathscr{A}(G/H)$ is the value of the functor that sends
$$G/H \mapsto K \{ G/G \leftarrow S \rightarrow G/H\} = K \{ S\rightarrow G/H\},$$
where $K$ denotes the group completion.
This is is the Burnside ring of $H$, since the category finite $G$-sets over $G/H$ is equivalent to the category of finite $H$-sets via the functor
$$\{ S\rightarrow G/H\} \rightarrow H\text{-Sets}$$
$$f\mapsto f^{-1}(*).$$

\vspace{5mm}

\begin{lemma}\label{BoxTimesVSBoxConnection}
The left Kan extension
$$L: Funct(\mathscr{O}_G^{Op}, Ab) \rightarrow Mackey_G$$
takes $\boxtimes$ to $\Box$, i.e.
$$L(A\boxtimes B) = L(A) \Box L(B)$$
for coefficient systems $A, B$.

\end{lemma}

\begin{proof}

We have left Kan extensions
$$L: Funct(\mathscr{O}_G^{Op}, Ab) \rightarrow Mackey_G = Add(\mathscr{B}_G, Ab)$$
and
$$\mathscr{L}: Funct(\mathscr{O}_G^{Op}\times \mathscr{O}_G^{Op}, Ab) \rightarrow BiAdd(\mathscr{B}_G \times \mathscr{B}_G, Ab)$$
where $Add$ denotes the category of additive functors and $BiAdd$ denotes the category of biadditive functors.
Then we can form a diagram
$$
\diagram
Funct(\mathscr{O}_G^{Op}, Ab)\times Funct(\mathscr{O}_G^{Op}, Ab) \rto^(0.575){\underline{\boxtimes}}\dto_{L\times L} &Funct(\mathscr{O}_G^{Op}\times \mathscr{O}_G^{Op}, Ab)\dto^{\mathscr{L}}\\
Add(\mathscr{B}_G, Ab)\times Add(\mathscr{B}_G, Ab) \rto_{\otimes} & BiAdd(\mathscr{B}_G\times\mathscr{B}_G, Ab)\\
\enddiagram
$$
This diagram commutes, since for coefficient systems $A, B$,
$$\mathscr{L}(A\otimes B) = \mathscr{B}_G\times \mathscr{B}_G \otimes_{\mathscr{O}_G^{Op}\times\mathscr{O}_G^{Op}}A\otimes B=L(A)\otimes L(B).$$
On the other hand, the diagram
$$
\diagram
Add(\mathscr{B}_G, Ab) \times Add(\mathscr{B}_G, Ab)\rto^(0.55)\otimes \drto_{\Box} & BiAdd(\mathscr{B}_G\times \mathscr{B}_G, Ab)\dto^{\times_{\#}} \\
 & Add(\mathscr{B}_G, Ab)
\enddiagram
$$
where $\times_{\#}$ the left Kan extension of $\times$, commutes since their right adjoints commute.
Also, we have a commutative diagram
$$
\diagram
Funct(\mathscr{O}_G^{Op} \times \mathscr{O}_G^{Op}, Ab) \rrto^(0.55){\times_{\#}}\dto_{\mathscr{L}} &  &Funct(\mathscr{O}_G^{Op}, Ab)\dto^L \\
BiAdd(\mathscr{B}_G \times \mathscr{B}_G. Ab) \rrto_{\times_{\#}} & & Add(\mathscr{B}_G, Ab).\\
\enddiagram
$$
Also,
$$
\diagram
Funct(\mathscr{O}_G^{Op}, Ab)\times Funct(\mathscr{O}_G^{Op}, Ab) \dto_{\underline{\boxtimes}} \drto^{\boxtimes}& \\
Funct(\mathscr{O}_G^{Op}\times \mathscr{O}_G^{Op}, Ab)\rto_(0.55){\times_{\#}} & Funct (\mathscr{O}_G^{Op}, Ab)\\
\enddiagram
$$
commutes by definition.

Therefore, by combining the diagrams, we get that
$$
\diagram
Funct(\mathscr{O}_G^{Op}, Ab)\times Funct(\mathscr{O}_G^{Op}, Ab)\dto_{L\times L} \rto^(0.65){\boxtimes}& Funct(\mathscr{O}_G^{Op}, Ab)\dto^L \\
Add(\mathscr{B}_G, Ab)\times Add(\mathscr{B}_G, Ab) \rto^(0.6){\Box}&Add(\mathscr{B}_G, Ab) \\
\enddiagram
$$
commutes.

\end{proof}

\begin{definition}
For a $G$-CW-complex $X$, define
$$C_M(X)= L C_G (X).$$
For a based $G$-CW-complex $X$, its based Mackey complex $\widetilde{C}_M(X)$ is defined analogously using $\widetilde{C}_n(X)$.
\end{definition}

A simpler construction may come to mind: sending $H \mapsto (C(X))^H$ also forms a Mackey functor.
However, this turns out to be the wrong construction for the present purpose.
We clarify the role of this construction (``constant coefficients in a looser sense") in Section \ref{AppendixMackeyChains}.

\begin{lemma}
For $G$-CW-complexes $X, Y$, 
$$C_M(X\times Y) = C_M(X) \Box C_M(Y).$$
For based $G$-CW-complexes $X, Y$,
$$\widetilde{C}_M(X\wedge Y) = \widetilde{C}_M(X) \Box \widetilde{C}_M(Y).$$
\end{lemma}

\begin{proof}
This follows immediately from Lemma \ref{CoefficientSystemBoxtimesConnection} and Lemma \ref{BoxTimesVSBoxConnection}.
The based case is analogous, omitting the base point cell.
\end{proof}

To give an example how these constructions are used, given an element $e_V$ of 
$\widetilde{H}_{m|V|}^G (S^{mV};\mathscr{A})$, one can use it to construct a map, for a based $G$-CW-complex $X$,
$$\widetilde{H}_k (X;M) \rightarrow \widetilde{H}^G_{k+m|V|}(X\wedge S^{mV}, M).$$
This can be described on chain level as follows. We have maps:
\beg{VerticalDiagramForIsomorphismForBorel}{
\diagram
\widetilde{H}_k (X;M) = H_k((\widetilde{C}_M(X) \Box M)(G/G))\dto^{\otimes e_V}\\
 H_k((\widetilde{C}_M (X)\Box M) (G/G) \otimes H_{m|V|} (\widetilde{C}_M(S^{mV})(G/G)\dto \\
H_{k+m|V|}((\widetilde{C}_M(X)\Box M)(G/G) \otimes \widetilde{C}_M(S^{mV})(G/G))\dto \\
H_{k+m|V|}( (\widetilde{C}_M(X) \Box M \Box \widetilde{C}_M(S^{mV}))(G/G)),\\
\enddiagram}
which is isomorphic to $\widetilde{H}^G_{k+m|V|}(X\wedge S^{mV}, M)$, since 
$$(\widetilde{C}_M(X) \Box M \Box \widetilde{C}_M(S^{mV}))(G/G) \cong \widetilde{C}_M(X\wedge S^{mV}) \Box M (G/G).$$

\vspace{5mm}

\section{Mackey Functors and Duality}\label{DualitySection}
 
\vspace{5mm}

We can also use these constructions to treat the $G$-equivariant duality
between ordinary homology and cohomology on chain level.

\begin{lemma}
$Mackey_G$ is a closed category, i.e. there exists a functor
$$Hom_M : Mackey_{G}^{Op} \times Mackey_G \rightarrow Mackey_G$$
such that we have a natural isomorphism
$$Mackey_G( M \Box N, P) \cong Mackey_G(M, Hom_M(N, P)).$$
\end{lemma}

\begin{proof}
Let $M, N, P$ be Mackey functors.
Suppose we have a natural transformation
$$M\Box N \rightarrow P.$$
By definition, this consists of morphisms, for finite $G$-sets $S,T$,
$$M(S)\otimes N(T) \rightarrow P(S\times T)$$
with naturality diagrams, for $\mathscr{B}_G$-morphisms $f: S\rightarrow S'$ and $\varphi: T\rightarrow T'$,
of the form
\beg{DiagramBoxTimesDuality}{
\diagram
M(S) \otimes N(T) \rto \dto &P(S\times T)\dto\\
M(S')\otimes N(T')\rto & P(S' \times T').\\
\enddiagram
}
By adjunction, we can write this as
$$M(S) \rightarrow (T\mapsto Hom_{Ab}(N(T), P(S\times T)).$$
By diagram \rref{DiagramBoxTimesDuality}, the system $g_T\in Hom_{Ab} (N(T), P(S\times T))$ satisfies
$$g_{T'} \circ N(\varphi) = P(S\times \varphi) \circ g_T.$$
Thus we can define
$$Hom_M(N, P)(S)=\{ g_T: N(T) \rightarrow P(S\times T)\; \mid$$ 
$$\text{ for } \varphi:T\rightarrow T', \; g_{T'} \circ N(\varphi ) = P (S\times \varphi)\circ g_T\}$$
Also, from diagram \rref{DiagramBoxTimesDuality}, we get a diagram
$$
\diagram
M(S)\dto\rto & (T\mapsto Hom_{Ab}(N(T), P(S\times T)))\dto\\
M(S')\rto & (T'\mapsto Hom_{Ab}(N(T'), P(S'\times T'))).
\enddiagram
$$
Thus, $Hom_M(N,P)$ is made into a Mackey functor by
$$f(T\rightarrow g_T): T\mapsto P(f\times T) \circ g_T.$$
Clearly, these choices are forced and reversible, thus proving the natural isomorphism.

\end{proof}

\begin{lemma}
Suppose $X$ is a $G$-CW-complex, then
$$Hom_M(C_M(X),\mathscr{A}) \cong R^*(C^*_G(X)),$$
where $R^*$ is the right Kan extension from co-coefficient systems to Mackey functors.
\end{lemma}

\begin{proof}
We will prove
$$Hom_M(L F_U, \mathscr{A})(S) \cong R^*F_U^*$$
for a general finite $G$-set $U$, extend to the case of infinite $G$-sets, and then apply this to $U= I_n$, in which case
$$F_U =  C_G(X)_n.$$
We have, for a $G$-set $S$,
\beg{DualityLemma}{Hom_M(LF_U, \mathscr{A})(S):T\mapsto g_T \in Hom_{Ab}(F_U(T), \mathscr{A}(S\times T)),}
where for $\varphi\in \mathscr{O}_G^{Op}(T,T')$,
$$g_{T'} \circ F_U(\varphi) = \mathscr{A}(S\times \varphi) \circ g_T.$$

Also, $F_U(T) = \mathscr{O}_G(T, U)$.
So, \rref{DualityLemma} is uniquely determined by
$$g_U(Id_U: U\rightarrow U) \in \mathscr{A}(S\times U).$$
So,
$$Hom_M(LF_U, \mathscr{A})(S) = \mathscr{A}(S\times U).$$

On the other hand, we will also show 
$$R^* (Hom(F_U,\Z))(S)= \mathscr{A}(S\times U).$$
We have
$$Hom(F_U, \Z): T\mapsto Hom(\Z\mathscr{O}_G (T,U),\Z).$$
Denote $\Gamma:= R^* Hom (F_U, \Z )$.
Then we have, by definition, 
$$\Gamma = Hom_{\mathscr{O}_G}(\mathscr{B}, Hom_{Ab} (F_U, \Z).$$
Thus, $\Gamma(T)$ is given by morphisms
$$\mathscr{B}_G(T,T')\otimes\Z \mathscr{O}_G(T',U) \rightarrow \Z$$
subject to the usual identification coming from $\mathscr{O}_G$-functors of $T'$.
Therefore
$$\Gamma(T) = Hom_{Ab}(\mathscr{B}_G(T,U),\Z)= Hom_{Ab}(\mathscr{A}(T\times U),\Z).$$

The proof is concluded by noting that the representable Mackey functor
$T\mapsto \mathscr{B}_G(U,T)=\mathscr{B}_G(T,U)=\mathscr{A}(T\times U)$ is isomorphic to $T\mapsto Hom_{Ab}(\mathscr{B}_G(U,T), \Z).$
This proves the claim for $U$ finite. For $U$ infinite, $C^*$, $R^*$ turn direct colimits into limits.

\end{proof}

\vspace{5mm}

\section{Mackey Chains with Constant Coefficients}\label{AppendixMackeyChains}

\vspace{5mm}

Since a chain complex is a type of a ``stable object," one could ask to what extent one can simply use chain complexes of $\Z[G]$-modules
(and their fixed points under subgroups) to generate our Mackey functors, instead of the more complicated structures treated earlier in this paper.
The answer is that a chain complex of $\Z[G]$-modules captures, essentially, information with constant Mackey functor coefficients,
rather than general Mackey coefficients.
The purpose of this section is to discuss this point.

\vspace{3mm}

We begin by characterizing Mackey functors which are modules over the Green functor $\underline{\Z}$.
(In this case, the Mackey functor structure determines the module structure.)

Next, we observe that fixed points and cofixed points of a $\Z[G]$-module give rise to a $\underline{\Z}$-module Mackey functor.
However, not all examples arise in this way.

On the other hand, a derived statement along such lines is true.
More precisely, we describe a notion of an equivalence on chain complexes of $\Z[G]$-modules based on quasi-isomorphisms of $H$-fixed points
(called fp-{\em equivalence})
and prove that the corresponding derived category is equivalent to the derived category of the abelian category of Mackey $\underline{\Z}$-modules.

There is also a seemingly symmetrical notion of equivalence based on quasi-isomorphisms on $H$-cofixed points (called cfp-{\em equivalence}) and
one may ask if it coincides with the equivalence based on fixed points.
We prove that this is in fact true for $G$ cyclic, but give an example showing that it is false for $G=\Z/2\times \Z/2$.

\vspace{5mm}

Our classification of $\underline{\Z}$-Mackey modules is given by the following
\begin{proposition}
The category of Mackey functor $\underline{\Z}$-modules is equivalent to the full subcategory of $Mackey_G$ on Mackey functors $M$
such that for $f: G/H \rightarrow G/K \in Mor(\mathscr{O}_G)$,
\beg{ConditionForModule}{f_* f^* = \frac{|K|}{|H|}.}
\end{proposition}

\begin{proof}

By definition, $M$ is a $\underline{\Z}$-module when, for a map 
$$f: G/H \rightarrow G/K,$$ 
the diagrams
$$
\diagram
M(G/H) \otimes \Z \rto^(0.55)= & M(G/H)\\
M(G/K) \otimes \Z \uto^{f^* \otimes f^*}\rto_(0.55)= & M(G/K) \uto_{f^*}
\enddiagram
$$

$$
\diagram
 &M(G/H) \otimes \Z \rto^= & M(G/H)\ddto^{f_*} \\
M(G/H) \otimes \Z \urto^{Id \otimes f^*}\drto_{f_*\otimes Id} & & \\
 & M(G/K) \otimes \Z \rto_= &  M(G/K)
\enddiagram
$$

$$
\diagram
 &M(G/H) \otimes \Z \rto^= & M(G/H)\ddto^{f_*} \\
M(G/K) \otimes \Z \urto^{f^* \otimes Id}\drto_{Id \otimes f_*} & & \\
 & M(G/K) \otimes \Z \rto_= &  M(G/K)
\enddiagram
$$
commute (see \cite{LiBoxProduct}, Lemma 21).
The conditions that the first and second diagrams commute are trivial. 
The third diagram gives the condition that
$$(f_* f^*: M(G/H) \rightarrow M(G/K)) = \frac{|K|}{|H|}.$$

\end{proof}

\vspace{5mm}

To describe the relation with $\Z[G]$-modules, note that for a $\Z[G]$-module $N$,
\beg{ExampleFixedPoints}{G/H \mapsto N^H}
\beg{ExampleDual}{G/H \mapsto N_H= N\otimes_{\Z[G]} \Z}
are Mackey functors satisfying \rref{ConditionForModule} where for \rref{ExampleFixedPoints}, restrictions are given by inclusions
and corestrictions are given by summing over cosets, and for \rref{ExampleDual}, vice versa. Thus, they give examples of $\underline{\Z}$-modules.
If $N=\Z$ (with trivial $G$-action), \rref{ExampleFixedPoints} gives $\underline{\Z}$ and \rref{ExampleDual} gives its dual.
This also shows that neither \rref{ExampleFixedPoints} nor \rref{ExampleDual} give all $\underline{\Z}$-modules.

\vspace{5mm}

On the other hand, we have the following
\begin{proposition}\label{MackeyAbelianConnection}
The functor $\Xi_{G/H}$ from the category of abelian groups to the category of Mackey functors defined by, for an object $A \in Ab$
and a subgroup $H\subseteq G$,
$$\Xi_{G/H}: A \mapsto (G/K \mapsto ( A((G/H)^K):= A\otimes \Z ((G/H)^K))),$$
(where restrictions are defined by restrictions of fixed points, and co-restrictions are given by sums over representatives of cosets)
is a universal $\underline{\Z}$-module on $F_{G/H}\otimes A$. Thus, $\Xi_{G/H}$ is left adjoint to evaluation of a $\underline{\Z}$-module
Mackey functor at $G/H$.

\end{proposition}

\begin{proof}

To simplify notation, we will just treat the case of $A= \Z$ (the general case is analogous).
We need to show that $\Xi_{G/H} (\Z)$ is the quotient $Q_{G/H}$ of the representable Mackey functor
$$G/K \mapsto \mathscr{B} (G/K, G/H)$$
modulo the relation \rref{ConditionForModule}.

To this end, note that $(\Z(G/H))^K$ is freely generated, as an abelian group, by
$$\sum_{\gamma\in K/(gHg^{-1})\cap K} 1\cdot (\gamma g H)$$
where $g$ runs through representatives of double cosets $K\backslash  G / H$.
Thus $\Z [G/H]^K$ is the free abelian group on generators of the form
\beg{LeftAdjointsLemmaDiagram}{
\diagram
 & G/(gHg^{-1} \cap K) \drto \dlto& \\
G/H & & G/K, \\
\enddiagram
}
or
\beg{FixedPointsAreFreeAbelianGroup}{\Xi_{G/H} (G/K) = (\Z(G/H))^K  = \Z \{ G/H \leftarrow G/(gHg^{-1} \cap K) \rightarrow G/K\}.}

On the other hand,
$$\mathscr{B}(G/K, G/H) = \Z \{ G/H \leftarrow G/J \rightarrow G/K | J\subseteq G/gHg^{-1} \cap K\}.$$
This can be written as
$$
\diagram
 & & S\dlto_q \drto^q & & \\
 & G/ gHg^{-1} \cap K \dlto& & G/ gHg^{-1} \cap K\drto & \\
G/H & & & & G/K \\
\enddiagram
$$
for some unique $q$.

By \rref{ConditionForModule}, this is identified with a multiple of \rref{LeftAdjointsLemmaDiagram}.
Thus $Q_{G/H}$ is a quotient of $\Xi_{G/H} (G/K)$ (see \rref{FixedPointsAreFreeAbelianGroup}).
On the other hand, $\Xi_{G/H}$ is a $\underline{\Z}$-module by the previous comment, and thus, the quotient map is the identity.

\end{proof}

\begin{corollary}\label{CorollaryUnivComplex}
For a $G$-CW-complex $X$,
$$C_{\underline{\Z}} : G/K \mapsto (C(X))^K$$
is the universal complex of $\underline{\Z}$-modules on the complex of coefficient systems $C_G (X)$.
\end{corollary}

\begin{proof}
Apply Proposition \ref{MackeyAbelianConnection} to the orbit summands of the set $I_n$ of $n$-cells (and $A= \Z$), and take direct sum.
\end{proof}

\noindent {\bf Comment:}

\noindent Thus, $C_{\underline{\Z}} (X)$ corresponds to the $G$-equivariant spectrum $X_+ \wedge H\underline{\Z}$.

\vspace{5mm}

Proposition \ref{MackeyAbelianConnection} and Corollary \ref{CorollaryUnivComplex}
can be used to show that a certain derived category of the category of chain complexes of
$\Z[G]$-modules is equivalence to the derived category of $\underline{\Z}$-Mackey modules. 
Denote by $\Z[G]\text{-Chain}$ the category of chain complexes of $\Z[G]$-modules.

\begin{definition}
A morphism $f:C\rightarrow D$ in $\Z[G]\text{-Chain}$ is called an {\em fp}-equivalence if for every subgroup $H\subseteq G$, the map induced on fixed points $f^H : C^H\rightarrow D^H$ is a quasi-isomorphism (i.e., induces an isomorphism in chain homology). Symmetrically, it is called
a {\em cfp}-equivalence if for every subgroup $H\subseteq G$, the map on cofixed points
$f_H :C_H \rightarrow D_H$
is a quasi-isomorphism.

\end{definition}

By Corollary \ref{CorollaryUnivComplex}, the homotopy category $h\Z[G]\text{-Chain}$ has colocalization with respect to fp-equivalences by cell chain complexes of $\Z[G]$-modules (see \cite{AGBook}, Section 5.2) where a cell chain complex of $\Z[G]$-modules $C$ is defined to be of the form
$$C=colim C_{(n)}$$
for some
$$C_{(-1)} \rightarrow C_{(0)} \rightarrow C_{(1)}\rightarrow \dots$$
where $C_{(-1)}=0$, and $C_{(n+1)}$ is the mapping cone of a chain map
$$P_{(n)} \rightarrow C_{(n)}$$
where $P_{(n)}$ is a direct sum of chain complexes of $\Z[G]$-modules of the form $\Z[G/H_i][n_i]$ and has 0 differential.
Thus we have proved
\begin{proposition}
The derived category of $\Z[G]\text{-Chain}$ with respect to fp-equivalences exists.
\end{proposition}
\qed

\begin{proposition}
The derived category $D_{fp}(G)$ of $\Z[G]\text{-Chain}$ with respect to fp-equivalences is equivalent to the
derived category $D\underline{\Z}\text{-Mod}$ of the abelian category $\underline{\Z}\text{-Mod}$ of Mackey modules over the constant Green functor
$\underline{\Z}$.
\end{proposition}

\begin{proof}
By construction, $D_{fg}(G)$ is a full subcategory of $D\underline{\Z}\text{-Mod}$.
The fully faithful functor is onto on isomorphism classes because every $D\underline{\Z}$-module
has a resolution in $D_{fp}(G)$ by Proposition \ref{MackeyAbelianConnection}.
\end{proof}

It is easy to see examples where a quasi-isomorphism of chain complex of $\Z[G]$-modules is not an fp-equivalence or a cfp-equivalence
(e.g., a free $\Z[\Z/2]$-resolution of $\Z$).

Because of the apparent symmetry, one can ask if there is a relationship between fp-equivalence and cfp-equivalence.
By the following proposition, they are actually the same when $G$ is cyclic.

\vspace{5mm}

\begin{proposition}
For a finite cyclic group $G$, a chain map of $\mathbb{Z}[G]$-modules is a quasi-isomorphism after taking $H$-fixed points for all subgroups
$H\subseteq G$ if and only if it is a quasi-isomorphism after taking $H$-cofixed points for all subgroups $H\subseteq G$.
\end{proposition}

\begin{proof}
First, a chain map
\beg{GeneralChainMapDiagram}{\diagram
\dots \rto & C_n \dto\rto & C_{n-1} \dto\rto & C_{n-2} \dto\rto & \dots\\
\dots \rto & D_n \rto & D_{n-1} \rto & D_{n-2} \rto & \dots\\
\enddiagram
}
is a quasi-isomorphism if and only if the totalization of \rref{GeneralChainMapDiagram} (considering it as a
double chain complex) is an exact sequence. Hence it suffices to show a sequence is long exact on the fixed points with respect to every subgroup
if and only if it is long exact on the cofixed points with respect to every subgroup.
Then, using the fact that fixed points (and cofixed points) are a left (right) exact functor on representations, this statement is equivalent to the statement
obtained by
replacing ``long exact" with ``short exact."

Now suppose we have a short exact sequence
$$0 \rightarrow K\rightarrow N \rightarrow M \rightarrow 0.$$
Then for every element $\alpha \in G$, we have
\vspace{5mm}

\hspace{15mm}\begin{tikzcd}
 & 0\arrow[d]& 0\arrow[d] &0 \arrow[d] & \\
0\arrow[r] & K^{\langle \alpha\rangle}\arrow[r] \arrow[d]& N^{\langle \alpha\rangle}\arrow[r]\arrow[d] & M^{\langle \alpha\rangle}\arrow[d]\arrow[llddd, controls={+(2,-1) and +(-2,1)}]& \\
0 \arrow[r] & K \arrow[r]\arrow[d]{d}{1-\alpha} & N \arrow[r] \arrow[d]{d}{1-\alpha} & M\arrow[r]\arrow[d]{d}{1-\alpha} & 0\\
0 \arrow[r] & K \arrow[r]\arrow[d] & N \arrow[r] \arrow[d] & M\arrow[r]\arrow[d] & 0\\
 & K_{\langle \alpha\rangle}\arrow[r] \arrow[d]& N_{\langle \alpha\rangle}\arrow[r]\arrow[d] & M_{\langle \alpha\rangle}\arrow[d]\arrow[r]& 0\\
 & 0 & 0 & 0 & \\
\end{tikzcd}

\noindent and the Snake Lemma gives a connecting map
$\gamma: M^{\langle \alpha \rangle} \rightarrow K_{\langle \alpha \rangle}$
making a six-term long exact sequence.

Then it follows that we have a short exact sequence
$$0 \rightarrow K^{\langle \alpha\rangle} \rightarrow N^{\langle \alpha \rangle} \rightarrow M^{\langle \alpha \rangle}\rightarrow 0$$
if and only if $\gamma: M^{\langle \alpha \rangle} \rightarrow K_{\langle \alpha \rangle }$ is 0, which also happens if and only if we have a
short exact sequence
$$0 \rightarrow K_{\langle \alpha \rangle} \rightarrow N_{\langle \alpha \rangle} \rightarrow M_{\langle \alpha \rangle} \rightarrow 0.$$
This proves the statement for $H= \langle \alpha \rangle$ for every $\alpha \in G$, hence implying it for every subgroup of $G$, since they are all of that form.

\end{proof}

One might ask if a similar statement is true for general finite groups $G$. However, this is false.
For a counterexample, consider the $\Z[\Z/2 \times \Z/2]$-module
\beg{Z2Z2RepWithRelation}{M= \Z[\Z/2 \times \Z/2] / (\alpha h - \alpha + h -1)}
where $h, \alpha $ are the generators of the two copies of $\Z/2$.

\vspace{5mm}

\begin{proposition}
There exists a sequence
\beg{CounterExSequence}{0 \rightarrow K \rightarrow N \rightarrow M \rightarrow 0}
which is short exact after taking fixed points with respect to any subgroup $H \subseteq \Z /2 \times \Z/2$ while
\beg{CofixedNotInjective}{K_{\Z/2\times \Z/2}\rightarrow N_{\Z/2\times \Z/2}}
is not injective.
\end{proposition}

\begin{proof}
First note that the relation of \rref{Z2Z2RepWithRelation} is preserved up to sign by the $\Z/2 \times \Z/2$-action,
and thus $M$ is a free $\Z$-module with basis $\{1, h, \alpha\}$.
Additionally, rationally, it splits as a sum of a fixed representation generated by $1+ \alpha$, an $\alpha h$-fixed sign representation generated by
$h-1$, and an $h$-fixed sign representation generated by $\alpha h -1$.
It follows that
$$M^{\Z/2\times \Z/2} = M^{\langle \alpha \rangle} = \Z \{1+\alpha \}$$
$$M^{\langle \alpha \rangle} = \Z \{ \alpha h-1, 1+\alpha\}$$
$$M^{\langle \alpha h\rangle} = \Z \{ h-1 , 1+\alpha\}.$$
Also,
$$\Z[\Z/2 \times \Z/2]^{\Z/2 \times \Z/2} = \Z\{1+\alpha +h + \alpha h\} $$
$$\Z[\Z/2\times \Z/2] ^{\langle h \rangle} = \Z\{1+h , \alpha + \alpha h\} $$
$$\Z [ \Z/2 \times \Z/2]^{\langle \alpha h \rangle} = \Z\{ 1+ \alpha h, \alpha + h\}.$$
Considering the map $\epsilon : \Z [\Z/2 \times \Z/2]\rightarrow M$ given by $\epsilon (1) =1 $, we have
$$\Z[\Z/2 \times \Z/2]^{\Z/2 \times \Z/2} \ni 1+ \alpha + h + \alpha h \mapsto 2(1+ \alpha)$$
$$\Z [\Z/2\times \Z/2]^{\langle h \rangle} \ni \alpha h + \alpha \mapsto (\alpha h-1) + (1+\alpha) $$
$$\Z[\Z/2 \times \Z/2]^{\langle\alpha h\rangle} \ni \alpha = h \mapsto (\alpha +1 ) + (h-1).$$
Thus, putting $N = \Z [\Z/2 \times \Z/2] \oplus \Z$ (where $\Z$ is the trivial $\Z/2 \times \Z/2$-representation), the map
$$\lambda: N \rightarrow M$$ 
given by $\lambda = (\epsilon, 1+\alpha)$ is onto on $H$-fixed points for all
$H\subseteq \Z/2\times \Z/2$.
Let $K= \text{\em Ker}(\lambda)$.
Since fixed points are a left exact functor on representations, \rref{CounterExSequence} is short exact after taking fixed points with
respect to every subgroup $H\subseteq \Z/2 \times \Z/2$.

Now, by construction,
$$K = \langle (\alpha h-\alpha + h-1, 0) \rangle \oplus \langle (1+ \alpha + h + \alpha h, -2)\rangle\subseteq N$$
where $\Z/2 \times \Z/2$ acts by the sign representation fixing $\alpha$ on the first summand and by the trivial representation on the
second summand.
Hence,
$$K_{\Z/2 \times \Z/2} \cong \Z/2 \oplus \Z$$
coming from the first and second summand, respectively.
Meanwhile,
$$N_{\Z/2\times \Z/2} \cong \Z \oplus \Z,$$
and thus \rref{CofixedNotInjective} cannot be injective.

\end{proof}

\end{document}